\documentclass[oneeqnum,onealgnum]{siamart190516}

\usepackage{amsfonts,amssymb}
\usepackage{graphicx}
\usepackage{epstopdf}
\usepackage{algorithm}
\usepackage{algorithmicx}
\usepackage{algpseudocode}
\algnewcommand\algorithmicinput{\textbf{Input:}}
\algnewcommand\Input{\item[\algorithmicinput]}
\algnewcommand\algorithmicoutput{\textbf{Output:}}
\algnewcommand\Output{\item[\algorithmicoutput]}
\ifpdf
  \DeclareGraphicsExtensions{.eps,.pdf,.png,.jpg}
\else
  \DeclareGraphicsExtensions{.eps}
\fi

\newcommand{\N}{\mathbb{N}}
\newcommand{\R}{\mathbb{R}}

\DeclareMathOperator{\im}{im}
\newcommand{\norm}[1]{\Vert #1 \Vert}
\newcommand{\ip}[2]{\langle #1, #2 \rangle}
\DeclareMathOperator*{\argmin}{argmin}
\DeclareMathOperator*{\lip}{Lip}
\DeclareMathOperator{\s}{s}

\newcommand{\setmapsto}{\multimap}

\newcommand{\proj}[2]{P_{#1}(#2)}

\makeatletter
\def\widebreve{\mathpalette\wide@breve}
\def\wide@breve#1#2{\sbox\z@{$#1#2$}%
     \mathop{\vbox{\m@th\ialign{##\crcr
\kern0.08em\brevefill#1{0.8\wd\z@}\crcr\noalign{\nointerlineskip}%
                    $\hss#1#2\hss$\crcr}}}\limits}
\def\brevefill#1#2{$\m@th\sbox\tw@{$#1($}%
  \hss\resizebox{#2}{\wd\tw@}{\rotatebox[origin=c]{90}{\upshape(}}\hss$}
\makeatletter

\newcommand{\gencone}[4]{{#1}_{#2}^{#4}(#3)} 
\newcommand{\tancone}[2]{\gencone{T}{#1}{#2}{}} 
\newcommand{\restancone}[2]{\gencone{\widebreve{T}}{#1}{#2}{}} 

\newcommand{\ushort}[1]{\mkern 0.9mu\underline{\mkern-0.9mu#1\mkern-0.9mu}\mkern 0.9mu}

\newcommand{\tp}{\top}
\DeclareMathOperator{\rank}{rank}
\newcommand{\st}{\mathrm{St}}

\newcommand{\ppgd}{\mathrm{P}^2\mathrm{GD}}
\newcommand{\ppgdr}{\mathrm{P}^2\mathrm{GDR}}
\newcommand{\rfd}{\mathrm{RFD}}
\newcommand{\rfdr}{\mathrm{RFDR}}

\newsiamremark{remark}{Remark}
\newsiamremark{hypothesis}{Hypothesis}
\crefname{hypothesis}{Hypothesis}{Hypotheses}
\newsiamthm{claim}{Claim}

\headers{An apocalypse-free retraction-free algorithm}{G. Olikier and P.-A. Absil}

\title{An apocalypse-free first-order low-rank optimization algorithm with at most one rank reduction attempt per iteration\thanks{This work was supported by the Fonds de la Recherche Scientifique -- FNRS and the Fonds Wetenschappelijk Onderzoek -- Vlaanderen under EOS Project no 30468160.}}

\author{Guillaume Olikier\thanks{ICTEAM Institute, UCLouvain, Avenue Georges Lema\^{\i}tre 4, 1348 Louvain-la-Neuve, Belgium (\email{guillaume.olikier@uclouvain.be}, \email{pa.absil@uclouvain.be}).}
\and P.-A. Absil\footnotemark[2]}

\usepackage{amsopn}

\begin{document}

\maketitle

\begin{abstract}
We consider the problem of minimizing a differentiable function with locally Lipschitz continuous gradient over the real determinantal variety, and present a first-order algorithm designed to find stationary points of that problem. This algorithm applies steps of a retraction-free descent method proposed by Schneider and Uschmajew (2015), while taking the numerical rank into account to attempt rank reductions. We prove that this algorithm produces a sequence of iterates the accumulation points of which are stationary, and therefore does not follow the so-called apocalypses described by Levin, Kileel, and Boumal (2022). Moreover, the rank reduction mechanism of this algorithm requires at most one rank reduction attempt per iteration, in contrast with the one of the $\ppgdr$ algorithm introduced by Olikier, Gallivan, and Absil (2022) which can require a number of rank reduction attempts equal to the rank of the iterate in the worst-case scenario.
\end{abstract}

\begin{keywords}
Convergence analysis $\cdot$ Stationarity $\cdot$ Low-rank optimization $\cdot$ Determinantal variety $\cdot$ First-order method $\cdot$ Tangent cones $\cdot$ QR factorization $\cdot$ Singular value decomposition.
\end{keywords}

\begin{AMS}
14M12, 65K10, 90C26, 90C30, 40A05.
\end{AMS}

\section{Introduction}
\label{sec:Introduction}
As in \cite{OlikierGallivanAbsil2022}, we consider the problem
\begin{equation}
\label{eq:LowRankOpti}
\min_{X \in \R_{\le r}^{m \times n}} f(X)
\end{equation}
of minimizing a differentiable function $f : \R^{m \times n} \to \R$ with locally Lipschitz continuous gradient over the \emph{determinantal variety} \cite[Lecture~9]{Harris}
\begin{equation}
\label{eq:RealDeterminantalVariety}
\R_{\le r}^{m \times n} := \{X \in \R^{m \times n} \mid \rank X \le r\},
\end{equation}
$m$, $n$, and $r$ being positive integers such that $r < \min\{m,n\}$.
This problem appears in several applications such as matrix equations, model reduction, matrix sensing, and matrix completion; see, e.g., \cite{SchneiderUschmajew2015,HaLiuBarber2020} and the references therein.
As problem \eqref{eq:LowRankOpti} is, in general, intractable \cite{GillisGlineur2011}, our goal is to find a \emph{stationary point} of this problem, i.e., a zero of the \emph{stationarity measure}
\begin{equation}
\label{eq:StationarityMeasure}
\s(\cdot; f, \R_{\le r}^{m \times n}) :
\R_{\le r}^{m \times n} \to \R :
X \mapsto \norm{\proj{\tancone{\R_{\le r}^{m \times n}}{X}}{-\nabla f(X)}}
\end{equation}
that returns the norm of any projection of $-\nabla f(X)$ onto the tangent cone to $\R_{\le r}^{m \times n}$ at $X$; the notation is introduced in Section~\ref{sec:Notation}. Indeed, by \cite[Lemmas~A.7 and A.8]{LevinKileelBoumal2022}, the correspondence
\begin{equation*}
\R_{\le r}^{m \times n} \setmapsto \R^{m \times n} : X \mapsto \proj{\tancone{\R_{\le r}^{m \times n}}{X}}{-\nabla f(X)}
\end{equation*}
depends on $f$ only through its restriction $f|_{\R_{\le r}^{m \times n}}$, and, by \cite[Theorem~6.12]{RockafellarWets} and \cite[Proposition~A.6]{LevinKileelBoumal2022}, if $X \in \R_{\le r}^{m \times n}$ is a local minimizer of $f|_{\R_{\le r}^{m \times n}}$, then $\s(X; f, \R_{\le r}^{m \times n}) = 0$.

To the best of our knowledge, the second-order method given in \cite[Algorithm~1]{LevinKileelBoumal2022} and the first-order method given in \cite[Algorithm~2]{OlikierGallivanAbsil2022}, dubbed $\ppgdr$, are the only two algorithms in the literature that provably converge to stationary points of \eqref{eq:LowRankOpti}. Other algorithms, such as \cite[Algorithm~3]{SchneiderUschmajew2015}, known as $\ppgd$, and \cite[Algorithm~4]{SchneiderUschmajew2015}, which we call $\rfd$ because it is a retraction-free descent method, can fail as they can produce a feasible sequence $(X_i)_{i \in \N}$ that converges to some point $X$ with the property that $\lim_{i \to \infty} \s(X_i; f, \R_{\le r}^{m \times n}) = 0 < \s(X; f, \R_{\le r}^{m \times n})$. Such a triplet $(X, (X_i)_{i \in \N}, f)$ is called an \emph{apocalypse} and the point $X$, which necessarily satisfies $\rank X < r$, is said to be \emph{apocalyptic} \cite[Definition~2.7]{LevinKileelBoumal2022}.

$\ppgdr$ is $\ppgd$ equipped with a rank reduction mechanism. On the one hand, this mechanism ensures the apocalypse-free property. On the other hand, as mentioned in \cite[\S 6]{OlikierGallivanAbsil2022}, it can generate a significant computational overhead in some situations. Indeed, its for-loop can consider up to $r$ rank reductions of the current iterate to each of which the $\ppgd$ map \cite[Algorithm~1]{OlikierGallivanAbsil2022} is applied. This is further discussed in Section~\ref{sec:PracticalImplementationComputationalCost}.

In this paper, we introduce a first-order algorithm on $\R_{\le r}^{m \times n}$ (Algorithm~\ref{algo:RFDR}), called $\rfdr$, that is apocalypse-free, in the sense that every accumulation point of the generated sequence is a stationary point (Theorem~\ref{thm:RFDRPolakConvergence}), which also implies that $\s(\cdot; f, \R_{\le r}^{m \times n})$ goes to zero along every convergent subsequence (Corollary~\ref{coro:RFDRPolakConvergence}).
$\rfdr$ is $\rfd$ equipped with a rank reduction mechanism.
$\rfdr$ presents two advantages over $\ppgdr$, which is the only other known apocalypse-free first-order optimization algorithm on $\R_{\le r}^{m \times n}$. First, it inherits the advantage that $\rfd$ has over $\ppgd$, namely being retraction-free; in other words, the updates are performed along a straight line. Second, its rank reduction mechanism is more efficient than the one of $\ppgdr$ because it ensures the apocalypse-free property by performing at most one rank reduction attempt per iteration (compare Algorithm~\ref{algo:RFDRmap} with \cite[Algorithm~3]{OlikierGallivanAbsil2022}).

This paper is organized as follows. After introducing the notation in Section~\ref{sec:Notation}, we define in Section~\ref{sec:RestrictedTangentConeDeterminantalVariety} the \emph{restricted tangent cone} to $\R_{\le r}^{m \times n}$; the descent direction used by $\rfd$ is the projection of the negative gradient onto that closed cone. In Section~\ref{sec:RFDmap}, we analyze the iteration map of $\rfd$ (Algorithm~\ref{algo:RFDmap}) under the assumption of local Lipschitz continuity of $\nabla f$. We introduce the $\rfdr$ algorithm in Section~\ref{sec:ProposedAlgorithm}, analyze its convergence properties in Section~\ref{sec:ConvergenceAnalysis}, and compare its computational cost with the one of $\ppgdr$ in Section~\ref{sec:PracticalImplementationComputationalCost}. Section~\ref{sec:Conclusion} contains concluding remarks.

\section{Notation}
\label{sec:Notation}
In this section, we introduce the notation used throughout the paper. This is the same notation as in \cite{OlikierGallivanAbsil2022} to which we refer for a more complete review of the background material.
The real vector space $\R^{m \times n}$ is endowed with the Frobenius inner product $\ip{\cdot}{\cdot}$, $\norm{\cdot}$ denotes the Frobenius norm, and, for all $X \in \R^{m \times n}$ and all $\rho \in (0,\infty)$, $B(X,\rho)$ and $B[X,\rho]$ respectively denote the open ball and the closed ball of center $X$ and radius $\rho$ in $\R^{m \times n}$.
Given a nonempty subset $\mathcal{S}$ of $\R^{m \times n}$, the tangent cone to $\mathcal{S}$ at $X \in \mathcal{S}$ is denoted by $\tancone{\mathcal{S}}{X}$, the distance from $X \in \R^{m \times n}$ to $\mathcal{S}$ is denoted by $d(X, \mathcal{S})$, and the projection of $X \in \R^{m \times n}$ onto $\mathcal{S}$ is denoted by $\proj{\mathcal{S}}{X}$.
For every $\ushort{r} \in \{0, \dots, \min\{m,n\}\}$,
\begin{equation*}
\R_{\ushort{r}}^{m \times n} := \{X \in \R^{m \times n} \mid \rank X = \ushort{r}\}
\end{equation*}
is the smooth manifold of $m \times n$ rank-$\ushort{r}$ real matrices, and, if $\ushort{r} \ge 1$,
\begin{equation*}
\st(\ushort{r},n) := \{U \in \R^{n \times \ushort{r}} \mid U^\tp U = I_{\ushort{r}}\}
\end{equation*}
is a Stiefel manifold.
We also write $\R_*^{m \times n} := \R_{\min\{m,n\}}^{m \times n}$.
The singular values of $X \in \R^{m \times n}$ are denoted by $\sigma_1(X) \ge \dots \ge \sigma_{\min\{m,n\}}(X) \ge 0$.

\section{The restricted tangent cone to the determinantal variety}
\label{sec:RestrictedTangentConeDeterminantalVariety}
In this section, we define the restricted tangent cone to $\R_{\le r}^{m \times n}$ (Definition~\ref{defi:RestrictedTangentConeDeterminantalVariety}) and prove that the descent direction introduced in \cite[\S 3.4]{SchneiderUschmajew2015} is the projection of the negative gradient onto that closed cone (Proposition~\ref{prop:ProjectionRestrictedTangentConeDeterminantalVariety}).
Throughout this section, $\ushort{r} \in \{1, \dots, r\}$, $X \in \R_{\ushort{r}}^{m \times n}$, $U \in \st(\ushort{r},m)$, $U_\perp \in \st(m-\ushort{r},m)$, $V \in \st(\ushort{r},n)$, $V_\perp \in \st(n-\ushort{r},n)$, $\im U = \im X$, $\im U_\perp = (\im X)^\perp$, $\im V = \im X^\tp$, and $\im V_\perp = (\im X^\tp)^\perp$.

The set introduced in the following definition is a closed cone contained in
\begin{equation*}
\tancone{\R_{\le r}^{m \times n}}{X}
= \left\{[U \; U_\perp] \begin{bmatrix} A & B \\ C & D \end{bmatrix} [V \; V_\perp]^\tp ~\bigg|~\begin{array}{l} A \in \R^{\ushort{r} \times \ushort{r}},\, B \in \R^{\ushort{r} \times n-\ushort{r}},\\ C \in \R^{m-\ushort{r} \times \ushort{r}},\, D \in \R_{\le r-\ushort{r}}^{m-\ushort{r} \times n-\ushort{r}} \end{array}\right\}.
\end{equation*}

\begin{definition}
\label{defi:RestrictedTangentConeDeterminantalVariety}
The \emph{restricted tangent cone} to $\R_{\le r}^{m \times n}$ at $X$ is
\begin{equation*}
\restancone{\R_{\le r}^{m \times n}}{X}
:= \left\{[U \; U_\perp] \begin{bmatrix} A & B \\ C & D \end{bmatrix} [V \; V_\perp]^\tp ~\bigg|~\begin{array}{l} A \in \R^{\ushort{r} \times \ushort{r}},\, B \in \R^{\ushort{r} \times n-\ushort{r}},\\ C \in \R^{m-\ushort{r} \times \ushort{r}},\, D \in \R_{\le r-\ushort{r}}^{m-\ushort{r} \times n-\ushort{r}}, \\ B = 0_{\ushort{r} \times n-\ushort{r}} \text{ or } C = 0_{m-\ushort{r} \times \ushort{r}} \end{array}\right\}.
\end{equation*}
Furthermore,
\begin{equation*}
\restancone{\R_{\le r}^{m \times n}}{0_{m \times n}}
:= \tancone{\R_{\le r}^{m \times n}}{0_{m \times n}}
= \R_{\le r}^{m \times n}.
\end{equation*}
\end{definition}

The following proposition can be compared to \cite[Proposition~2.7]{OlikierGallivanAbsil2022}.

\begin{proposition}
\label{prop:ProjectionRestrictedTangentConeDeterminantalVariety}
Let $Z \in \R^{m \times n}$ be written as
\begin{equation*}
Z = [U \; U_\perp] \begin{bmatrix} A & B \\ C & D \end{bmatrix} [V \; V_\perp]^\tp
\end{equation*}
with $A = U^\tp Z V$, $B = U^\tp Z V_\perp$, $C = U_\perp^\tp Z V$, and $D = U_\perp^\tp Z V_\perp$. Then,
\begin{equation}
\label{eq:ProjectionRestrictedTangentConeDeterminantalVariety}
\proj{\restancone{\R_{\le r}^{m \times n}}{X}}{Z}
= \left\{\begin{array}{ll}
\begin{array}{l}
[U \; U_\perp] \begin{bmatrix} A & B \\ 0_{m-\ushort{r} \times \ushort{r}} & \proj{\R_{\le r-\ushort{r}}^{m-\ushort{r} \times n-\ushort{r}}}{D} \end{bmatrix} [V \; V_\perp]^\tp
\end{array} & \text{if } \norm{B} > \norm{C},\\
\begin{array}{l}
[U \; U_\perp] \begin{bmatrix} A & B \\ 0_{m-\ushort{r} \times \ushort{r}} & \proj{\R_{\le r-\ushort{r}}^{m-\ushort{r} \times n-\ushort{r}}}{D} \end{bmatrix} [V \; V_\perp]^\tp\\
\cup [U \; U_\perp] \begin{bmatrix} A & 0_{\ushort{r} \times n-\ushort{r}} \\ C & \proj{\R_{\le r-\ushort{r}}^{m-\ushort{r} \times n-\ushort{r}}}{D} \end{bmatrix} [V \; V_\perp]^\tp
\end{array} & \text{if } \norm{B} = \norm{C},\\
\begin{array}{l}
[U \; U_\perp] \begin{bmatrix} A & 0_{\ushort{r} \times n-\ushort{r}} \\ C & \proj{\R_{\le r-\ushort{r}}^{m-\ushort{r} \times n-\ushort{r}}}{D} \end{bmatrix} [V \; V_\perp]^\tp
\end{array} & \text{if } \norm{B} < \norm{C}.
\end{array}\right.
\end{equation}
Moreover,
\begin{equation}
\label{eq:NormProjectionRestrictedTangentConeDeterminantalVariety}
\norm{\proj{\tancone{\R_{\le r}^{m \times n}}{X}}{Z}} \ge
\norm{\proj{\restancone{\R_{\le r}^{m \times n}}{X}}{Z}} \ge
\frac{1}{\sqrt{2}} \norm{\proj{\tancone{\R_{\le r}^{m \times n}}{X}}{Z}}
\end{equation}
and, for all $\breve{Z} \in \proj{\restancone{\R_{\le r}^{m \times n}}{X}}{Z}$,
\begin{equation}
\label{eq:InnerProductApproximateProjectionTangentConeDeterminantalVariety}
\ip{Z}{\breve{Z}} = \norm{\breve{Z}}^2.
\end{equation}
Furthermore,
\begin{equation}
\label{eq:NormProjectionTangentConeDeterminantalVariety}
\norm{Z}
\ge \norm{\proj{\tancone{\R_{\le r}^{m \times n}}{X}}{Z}}
\ge \sqrt{\frac{r-\ushort{r}}{\min\{m,n\}-\ushort{r}}} \norm{Z}.
\end{equation}
\end{proposition}

\begin{proof}
We prove only \eqref{eq:ProjectionRestrictedTangentConeDeterminantalVariety} as \eqref{eq:NormProjectionRestrictedTangentConeDeterminantalVariety} and \eqref{eq:InnerProductApproximateProjectionTangentConeDeterminantalVariety} are given in \cite[\S 3.4]{SchneiderUschmajew2015} and \eqref{eq:NormProjectionTangentConeDeterminantalVariety} in \cite[(3.7)]{SchneiderUschmajew2015}. Let $\breve{Z} \in \restancone{\R_{\le r}^{m \times n}}{X}$. Then,
\begin{equation*}
\breve{Z} = [U \; U_\perp] \begin{bmatrix} \breve{A} & \breve{B} \\ \breve{C} & \breve{D} \end{bmatrix} [V \; V_\perp]^\tp
\end{equation*}
for some $\breve{A} \in \R^{\ushort{r} \times \ushort{r}}$, $\breve{B} \in \R^{\ushort{r} \times n-\ushort{r}}$, $\breve{C} \in \R^{m-\ushort{r} \times \ushort{r}}$, and $\breve{D} \in \R_{\le r-\ushort{r}}^{m-\ushort{r} \times n-\ushort{r}}$ such that $\breve{B} = 0_{\ushort{r} \times n-\ushort{r}}$ or $\breve{C} = 0_{m-\ushort{r} \times \ushort{r}}$. Thus,
\begin{align*}
\norm{Z-\breve{Z}}^2
&= \left\|\begin{bmatrix} A-\breve{A} & B-\breve{B} \\ C-\breve{C} & D-\breve{D} \end{bmatrix}\right\|^2\\
&= \norm{A-\breve{A}}^2 + \norm{B-\breve{B}}^2 + \norm{C-\breve{C}}^2 + \norm{D-\breve{D}}^2\\
&\ge 0 + \min\{\norm{B}, \norm{C}\}^2 + d(D, \R_{\le r-\ushort{r}}^{m-\ushort{r} \times n-\ushort{r}})^2.
\end{align*}
Therefore,
\begin{equation*}
d(Z, \restancone{\R_{\le r}^{m \times n}}{X})^2 \ge \min\{\norm{B}, \norm{C}\}^2 + d(D, \R_{\le r-\ushort{r}}^{m-\ushort{r} \times n-\ushort{r}})^2,
\end{equation*}
a bound which is reached only by \eqref{eq:ProjectionRestrictedTangentConeDeterminantalVariety}.
\end{proof}

In practice, the projection onto $\tancone{\R_{\le r}^{m \times n}}{X}$ can be computed thanks to \cite[Algorithm~2]{SchneiderUschmajew2015}. This does not rely on $U_\perp$ and $V_\perp$, which are huge in the frequently encountered case where $r \ll \min\{m,n\}$. The practical computation of the projection onto $\restancone{\R_{\le r}^{m \times n}}{X}$ is described in \cite[\S 3.4]{SchneiderUschmajew2015} and does not rely on $U_\perp$ and $V_\perp$ either. We discuss the practical implementation and the computational cost in Section~\ref{sec:PracticalImplementationComputationalCost}.

\section{The $\rfd$ map}
\label{sec:RFDmap}
In this section, we analyze Algorithm~\ref{algo:RFDmap}---which corresponds to the iteration map of $\rfd$ \cite[Algorithm~4]{SchneiderUschmajew2015} except that the initial step size for the backtracking procedure is chosen in a given bounded interval---under the assumption that $f$ is differentiable with $\nabla f$ locally Lipschitz continuous. This serves as a basis for the convergence analysis conducted in Section~\ref{sec:ConvergenceAnalysis} since the $\rfd$ map (Algorithm~\ref{algo:RFDmap}) is used as a subroutine by the $\rfdr$ map (Algorithm~\ref{algo:RFDRmap}).
The analysis conducted in this section is to $\rfd$ as the analysis conducted in \cite[\S 3]{OlikierGallivanAbsil2022} is to $\ppgd$, although it is simpler since $\rfd$ does not require any retraction as $X + \alpha G \in \R_{\le r}^{m \times n}$ for all $X \in \R_{\le r}^{m \times n}$, all $\alpha \in (0,\infty)$, and all $G \in \restancone{\R_{\le r}^{m \times n}}{X}$.

\begin{algorithm}[H]
\caption{$\rfd$ map (based on \cite[Algorithm~4]{SchneiderUschmajew2015})}
\label{algo:RFDmap}
\begin{algorithmic}[1]
\Require
$(f, r, \ushort{\alpha}, \bar{\alpha}, \beta, c)$ where $f : \R^{m \times n} \to \R$ is differentiable with $\nabla f$ locally Lipschitz continuous, $r < \min\{m,n\}$ is a positive integer, $0 < \ushort{\alpha} \le \bar{\alpha} < \infty$, and $\beta, c \in (0,1)$.
\Input
$X \in \R_{\le r}^{m \times n}$ such that $\s(X; f, \R_{\le r}^{m \times n}) > 0$.
\Output
a point in $\rfd(X; f, r, \ushort{\alpha}, \bar{\alpha}, \beta, c)$.

\State
Choose $G \in \proj{\restancone{\R_{\le r}^{m \times n}}{X}}{-\nabla f(X)}$ and $\alpha \in [\ushort{\alpha},\bar{\alpha}]$;
\label{algo:RFDmap:FirstLine}
\While
{$f(X + \alpha G) > f(X) - c \, \alpha \norm{G}^2$}
\State
$\alpha \gets \alpha \beta$;
\EndWhile
\State
Return $X + \alpha G$.
\end{algorithmic}
\end{algorithm}

Let us recall that, since $\nabla f$ is locally Lipschitz continuous, for every closed ball $\mathcal{B} \subset \R^{m \times n}$,
\begin{equation*}
\lip_{\mathcal{B}}(\nabla f) := \sup_{\substack{X, Y \in \mathcal{B} \\ X \ne Y}} \frac{\norm{\nabla f(X) - \nabla f(Y)}}{\norm{X-Y}} < \infty,
\end{equation*}
which implies, by \cite[Lemma~1.2.3]{Nesterov2018}, that, for all $X, Y \in \mathcal{B}$,
\begin{equation}
\label{eq:InequalityLipschitzContinuousGradient}
|f(Y) - f(X) - \ip{\nabla f(X)}{Y-X}| \le \frac{\lip_{\mathcal{B}}(\nabla f)}{2} \norm{Y-X}^2.
\end{equation}

\begin{proposition}
\label{prop:RFDmapUpperBoundCost}
Let $X \in \R_{\le r}^{m \times n}$ and $\bar{\alpha} \in (0,\infty)$.
Let $\mathcal{B} \subset \R^{m \times n}$ be a closed ball such that, for all $G \in \proj{\restancone{\R_{\le r}^{m \times n}}{X}}{-\nabla f(X)}$ and all $\alpha \in [0,\bar{\alpha}]$, $X + \alpha G \in \mathcal{B}$; an example of such a ball is $B[X, \bar{\alpha}\s(X; f, \R_{\le r}^{m \times n})]$.
Then, for all $G \in \proj{\restancone{\R_{\le r}^{m \times n}}{X}}{-\nabla f(X)}$ and all $\alpha \in [0,\bar{\alpha}]$,
\begin{equation}
\label{eq:RFDmapUpperBoundCost}
f(X + \alpha G) \le f(X) + \norm{G}^2 \alpha \left(-1+\frac{\lip_\mathcal{B}(\nabla f)}{2}\alpha\right).
\end{equation}
\end{proposition}

\begin{proof}
The example $B[X, \bar{\alpha}\s(X; f, \R_{\le r}^{m \times n})]$ is correct by the first inequality of \eqref{eq:NormProjectionRestrictedTangentConeDeterminantalVariety}.
The inequality \eqref{eq:RFDmapUpperBoundCost} is based on \eqref{eq:InequalityLipschitzContinuousGradient} and \eqref{eq:InnerProductApproximateProjectionTangentConeDeterminantalVariety}:
\begin{align*}
f(X+\alpha G)-f(X)
&\le \ip{\nabla f(X)}{(X+\alpha G)-X} + \frac{\lip_\mathcal{B}(\nabla f)}{2} \norm{(X+\alpha G)-X}^2\\
&= -\alpha \norm{G}^2 + \frac{\lip_\mathcal{B}(\nabla f)}{2} \alpha^2 \norm{G}^2.
\end{align*}
\end{proof}

Let us make two remarks concerning the preceding proposition. First, the existence of a ball $\mathcal{B}$ crucially relies on the upper bound $\bar{\alpha}$ required by Algorithm~\ref{algo:RFDmap}. Second, in contrast with \cite[(8)]{OlikierGallivanAbsil2022}, the upper bound \eqref{eq:RFDmapUpperBoundCost} does not depend on the curvature of a fixed-rank manifold. In particular, \eqref{eq:RFDmapUpperBoundCost} does not involve any singular value. This fundamental difference is the reason why $\rfdr$ is apocalypse-free while requiring at most one rank reduction attempt per iteration, whereas $\ppgdr$ can require up to $r$ attempts.

Observe that the $\frac{1}{2}$ factor in the sufficient decrease condition of the following result does not appear in the Armijo condition given in \cite[Corollary~3.2]{OlikierGallivanAbsil2022}.

\begin{corollary}
\label{coro:RFDmapCostDecrease}
Every $\tilde{X} \in \hyperref[algo:RFDmap]{\rfd}(X; f, r, \ushort{\alpha}, \bar{\alpha}, \beta, c)$ satisfies the sufficient decrease condition
\begin{equation}
\label{eq:RFDmapCostDecrease}
f(\tilde{X}) \le f(X) - \frac{1}{2} c \, \alpha \s(X; f, \R_{\le r}^{m \times n})^2
\end{equation}
for some $\alpha \in \big[\min\{\ushort{\alpha}, 2\beta\frac{1-c}{\lip_\mathcal{B}(\nabla f)}\}, \bar{\alpha}\big]$, where $\mathcal{B}$ is any closed ball as in Proposition~\ref{prop:RFDmapUpperBoundCost}. Moreover, the number of iterations in the while loop is at most
\begin{equation}
\label{eq:RFDmapMaxNumIterationsWhile}
\max\left\{0, \left\lfloor\ln\left(\frac{2\beta(1-c)}{\alpha_0\lip_\mathcal{B}(\nabla f)}\right)/\ln \beta\right\rfloor\right\},
\end{equation}
where $\alpha_0$ is the initial step size chosen in line~\ref{algo:RFDmap:FirstLine}.
\end{corollary}

\begin{proof}
For all $\alpha \in (0,\infty)$,
\begin{equation*}
f(X) + \norm{G}^2 \alpha \left(-1+\frac{\lip_\mathcal{B}(\nabla f)}{2}\alpha\right)
\le f(X) - c \norm{G}^2 \alpha
\quad \text{iff} \quad \alpha \le 2\frac{1-c}{\lip_\mathcal{B}(\nabla f)}.
\end{equation*}
Since the left-hand side of the first inequality is an upper bound on $f(X + \alpha G)$ for all $\alpha \in (0,\bar{\alpha}]$, the condition for the while loop to stop is necessarily satisfied if $\alpha \in (0,\min\{\bar{\alpha}, 2\frac{1-c}{\lip_\mathcal{B}(\nabla f)}\}]$.
Therefore, either the initial step size chosen in $[\ushort{\alpha},\bar{\alpha}]$ satisfies that condition or the while loop ends with $\alpha$ such that $\frac{\alpha}{\beta} > 2\frac{1-c}{\lip_\mathcal{B}(\nabla f)}$. The sufficient decrease condition \eqref{eq:RFDmapCostDecrease} then follows from the second inequality of \eqref{eq:NormProjectionRestrictedTangentConeDeterminantalVariety}. The upper bound \eqref{eq:RFDmapMaxNumIterationsWhile} follows from the same observation.
\end{proof}

\section{The proposed algorithm}
\label{sec:ProposedAlgorithm}
We now introduce a first-order optimization algorithm on $\R_{\le r}^{m \times n}$ called $\rfdr$. The iteration map of this algorithm, called the $\rfdr$ map, is defined as Algorithm~\ref{algo:RFDRmap}. Given $X \in \R_{\le r}^{m \times n}$ as input, the $\rfdr$ map applies the $\rfd$ map (Algorithm~\ref{algo:RFDmap}) to $X$, thereby producing a point $\tilde{X}$, but also, if $\sigma_r(X)$ is positive but smaller than some treshold $\Delta \in (0,\infty)$, to a projection $\hat{X}$ of $X$ onto $\R_{r-1}^{m \times n}$, then producing a point $\tilde{X}^\mathrm{R}$, and outputs the point among $\tilde{X}$ and $\tilde{X}^\mathrm{R}$ that decreases $f$ the most.

\begin{algorithm}[H]
\caption{$\rfdr$ map}
\label{algo:RFDRmap}
\begin{algorithmic}[1]
\Require
$(f, r, \ushort{\alpha}, \bar{\alpha}, \beta, c, \Delta)$ where $f : \R^{m \times n} \to \R$ is differentiable with $\nabla f$ locally Lipschitz continuous, $r < \min\{m,n\}$ is a positive integer, $0 < \ushort{\alpha} \le \bar{\alpha} < \infty$, $\beta, c \in (0,1)$, and $\Delta \in (0,\infty)$.
\Input
$X \in \R_{\le r}^{m \times n}$ such that $\s(X; f, \R_{\le r}^{m \times n}) > 0$.
\Output
a point in $\rfdr(X; f, r, \ushort{\alpha}, \bar{\alpha}, \beta, c, \Delta)$.

\State
Choose $\tilde{X} \in \hyperref[algo:RFDmap]{\rfd}(X; f, r, \ushort{\alpha}, \bar{\alpha}, \beta, c)$;
\If
{$\sigma_r(X) \in (0,\Delta]$}
\State
Choose $\hat{X} \in \proj{\R_{r-1}^{m \times n}}{X}$;
\State
Choose $\tilde{X}^\mathrm{R} \in \hyperref[algo:RFDmap]{\rfd}(\hat{X}; f, r, \ushort{\alpha}, \bar{\alpha}, \beta, c)$;
\State
Return $Y \in \argmin_{\{\tilde{X}, \tilde{X}^\mathrm{R}\}} f$.
\Else
\State
Return $\tilde{X}$.
\EndIf
\end{algorithmic}
\end{algorithm}

The $\rfdr$ algorithm is defined as Algorithm~\ref{algo:RFDR}. It produces a sequence along which $f$ is strictly decreasing. As announced in the second remark following Proposition~\ref{prop:RFDmapUpperBoundCost}, it turns out that the mere monitoring of the $r$th singular value of the iterates is enough to make the algorithm apocalypse-free, which we prove in the next section.

\begin{algorithm}[H]
\caption{Iterative $\rfdr$}
\label{algo:RFDR}
\begin{algorithmic}[1]
\Require
$(X_0, f, r, \ushort{\alpha}, \bar{\alpha}, \beta, c, \Delta)$ where $X_0 \in \R_{\le r}^{m \times n}$, $f : \R^{m \times n} \to \R$ is differentiable with $\nabla f$ locally Lipschitz continuous, $r < \min\{m,n\}$ is a positive integer, $0 < \ushort{\alpha} \le \bar{\alpha} < \infty$, $\beta, c \in (0,1)$, and $\Delta \in (0,\infty)$.

\State
$i \gets 0$;
\While
{$\s(X_i; f, \R_{\le r}^{m \times n}) > 0$}
\State
Choose $X_{i+1} \in \hyperref[algo:RFDRmap]{\rfdr}(X_i; f, r, \ushort{\alpha}, \bar{\alpha}, \beta, c, \Delta)$;
\State
$i \gets i+1$;
\EndWhile
\end{algorithmic}
\end{algorithm}

\section{Convergence analysis}
\label{sec:ConvergenceAnalysis}
The purpose of this section is to prove Theorem~\ref{thm:RFDRPolakConvergence}. To this end, we use the abstract framework proposed in \cite[\S 1.3]{Polak1971}.
Indeed, the problem considered in this paper can be formulated as follows: find $X \in \R_{\le r}^{m \times n}$ such that $\s(X; f, \R_{\le r}^{m \times n}) = 0$. It is thus a particular instance of \cite[Abstract Problem~1]{Polak1971} where the Banach space is $\R^{m \times n}$, its closed subset is $\R_{\le r}^{m \times n}$, and ``desirable'' means stationary.
Moreover, $\rfdr$ (Algorithm~\ref{algo:RFDR}) is a particular instance of \cite[Algorithm Model~9]{Polak1971} where the ``stop rule'' is $f$ and the ``search function'' is the $\rfdr$ map (Algorithm~\ref{algo:RFDRmap}).

Thus, to prove Theorem~\ref{thm:RFDRPolakConvergence}, it suffices to verify that $\rfdr$ satisfies the two assumptions of \cite[Theorem~10]{Polak1971}, which we do below.

The first assumption is that the objective function $f$ is continuous at each nondesirable point or bounded from below on $\R_{\le r}^{m \times n}$. It is thus satisfied since $f$ is continuous. The following proposition states that the second assumption is satisfied.

\begin{proposition}
\label{prop:RFDRPolak}
For every $\ushort{X} \in \R_{\le r}^{m \times n}$ such that $\s(\ushort{X}; f, \R_{\le r}^{m \times n}) > 0$, there exist $\varepsilon(\ushort{X}), \delta(\ushort{X}) \in (0,\infty)$ such that, for all $X \in B[\ushort{X},\varepsilon(\ushort{X})] \cap \R_{\le r}^{m \times n}$ and all $Y \in \hyperref[algo:RFDRmap]{\rfdr}(X; f, r, \ushort{\alpha}, \bar{\alpha}, \beta, c, \Delta)$,
\begin{equation}
\label{eq:PolakUniformSufficientDecrease}
f(Y) - f(X) \le - \delta(\ushort{X}).
\end{equation}
\end{proposition}

\begin{proof}
Let $\ushort{X} \in \R_{\le r}^{m \times n}$ be such that $\s(\ushort{X}; f, \R_{\le r}^{m \times n}) > 0$. This proof constructs $\varepsilon(\ushort{X})$ and $\delta(\ushort{X})$ based on the sufficient decrease condition \eqref{eq:RFDmapCostDecrease} given in Corollary~\ref{coro:RFDmapCostDecrease}. This requires to derive local lower and upper bounds on $\s(\cdot; f, \R_{\le r}^{m \times n})$ around $\ushort{X}$. It first considers the case where $\rank \ushort{X} = r$, in which the construction essentially relies on the continuity of $\s(\cdot; f, \R_{\le r}^{m \times n})$ on $B(\ushort{X},\sigma_r(\ushort{X})) \cap \R_{\le r}^{m \times n}$. Then, it focuses on the case where $\rank \ushort{X} < r$, in which the bounds on $\s(\cdot; f, \R_{\le r}^{m \times n})$ follow from the bounds \eqref{eq:RFDRPolakInequalityNormGradient} on $\nabla f$ thanks to \eqref{eq:NormProjectionTangentConeDeterminantalVariety}. If $\rank X < r$, then the second inequality of \eqref{eq:NormProjectionTangentConeDeterminantalVariety} readily gives a lower bound on $\s(X; f, \R_{\le r}^{m \times n})$. This is not the case if $\rank X = r$, however. This is where the rank reduction mechanism comes into play. It considers a projection $\hat{X}$ of $X$ onto $\R_{r-1}^{m \times n}$, and the second inequality of \eqref{eq:NormProjectionTangentConeDeterminantalVariety} gives a lower bound on $\s(\hat{X}; f, \R_{\le r}^{m \times n})$. The inequality \eqref{eq:PolakUniformSufficientDecrease} is then obtained from \eqref{eq:RFDRPolakCostContinuous} which follows from the continuity of $f$ at $\ushort{X}$.

Let us first consider the case where $\rank \ushort{X} = r$. On $B(\ushort{X},\sigma_r(\ushort{X})) \cap \R_{\le r}^{m \times n} = B(\ushort{X},\sigma_r(\ushort{X})) \cap \R_r^{m \times n}$, $\s(\cdot; f, \R_{\le r}^{m \times n})$ coincides with the norm of the Riemannian gradient of the restriction of $f$ to the smooth manifold $\R_r^{m \times n}$, which is continuous. In particular, there exists $\rho(\ushort{X}) \in (0,\sigma_r(\ushort{X}))$ such that $\s(B[\ushort{X},\rho(\ushort{X})] \cap \R_r^{m \times n}; f, \R_{\le r}^{m \times n}) \subseteq [\frac{1}{2}\s(\ushort{X}; f, \R_{\le r}^{m \times n}), \frac{3}{2}\s(\ushort{X}; f, \R_{\le r}^{m \times n})]$. Let $\bar{\rho}(\ushort{X}) := \rho(\ushort{X}) + \frac{3}{2} \bar{\alpha} \s(\ushort{X}; f, \R_{\le r}^{m \times n})$, $\varepsilon(\ushort{X}) := \rho(\ushort{X})$, and
\begin{equation*}
\delta(\ushort{X}) := \frac{1}{8} c \s(\ushort{X}; f, \R_{\le r}^{m \times n})^2 \min\left\{\ushort{\alpha}, 2\beta\frac{1-c}{\lip_{B[\ushort{X},\bar{\rho}(\ushort{X})]}(\nabla f)}\right\}.
\end{equation*}
Let $X \in B[\ushort{X},\varepsilon(\ushort{X})] \cap \R_{\le r}^{m \times n}$. Then, $B[X, \bar{\alpha} \s(X; f, \R_{\le r}^{m \times n})] \subseteq B[\ushort{X}, \bar{\rho}(\ushort{X})]$. Indeed, for all $Z \in B[X, \bar{\alpha} \s(X; f, \R_{\le r}^{m \times n})]$,
\begin{equation*}
\norm{Z-\ushort{X}}
\le \norm{Z-X} + \norm{X-\ushort{X}}
\le \bar{\alpha} \s(X; f, \R_{\le r}^{m \times n}) + \rho(\ushort{X})
\le \bar{\rho}(\ushort{X}).
\end{equation*}
Therefore, Corollary~\ref{coro:RFDmapCostDecrease} applies and, by \eqref{eq:RFDmapCostDecrease}, for all $\tilde{X} \in \hyperref[algo:RFDmap]{\rfd}(X; f, r, \ushort{\alpha}, \bar{\alpha}, \beta, c)$, 
\begin{equation*}
f(\tilde{X})
\le f(X) - \frac{1}{2} c \s(X; f, \R_{\le r}^{m \times n})^2 \min\left\{\ushort{\alpha}, 2\beta\frac{1-c}{\lip_{B[\ushort{X},\bar{\rho}(\ushort{X})]}(\nabla f)}\right\}
\le f(X) - \delta(\ushort{X}),
\end{equation*}
and thus, for all $Y \in \hyperref[algo:RFDRmap]{\rfdr}(X; f, r, \ushort{\alpha}, \bar{\alpha}, \beta, c, \Delta)$,
\begin{equation*}
f(Y)
\le f(\tilde{X})
\le f(X) - \delta(\ushort{X}).
\end{equation*}
This completes the proof for the case where $\rank \ushort{X} = r$.

Let us now consider the case where $\rank \ushort{X} < r$. By the first inequality of \eqref{eq:NormProjectionTangentConeDeterminantalVariety}, $\norm{\nabla f(\ushort{X})} \ge \s(\ushort{X}; f, \R_{\le r}^{m \times n})$. Let $\bar{\rho}(\ushort{X}) := \Delta + \frac{3}{2} \bar{\alpha} \norm{\nabla f(\ushort{X})}$ and
\begin{equation}
\label{eq:RFDRPolakDelta}
\delta(\ushort{X}) := \frac{1}{24} c \frac{\norm{\nabla f(\ushort{X})}^2}{\min\{m,n\}-r+1} \min\left\{\ushort{\alpha}, 2\beta\frac{1-c}{\lip_{B[\ushort{X},\bar{\rho}(\ushort{X})]}(\nabla f)}\right\}.
\end{equation}
Since $f$ is continuous at $\ushort{X}$, there exists $\rho_f(\ushort{X}) \in (0,\infty)$ such that $f(B[\ushort{X},\rho_f(\ushort{X})]) \subseteq [f(\ushort{X})-\delta(\ushort{X}), f(\ushort{X})+\delta(\ushort{X})]$. Let $\varepsilon(\ushort{X}) := \frac{1}{2} \min\Big\{\Delta, \rho_f(\ushort{X}), \frac{\norm{\nabla f(\ushort{X})}}{2\lip_{B[\ushort{X},\Delta]}(\nabla f)}\Big\}$. Then, for all $Z \in B[\ushort{X}, 2\varepsilon(\ushort{X})]$, since
\begin{align*}
|\norm{\nabla f(Z)}-\norm{\nabla f(\ushort{X})}|
&\le \norm{\nabla f(Z)-\nabla f(\ushort{X})}\\
&\le \lip_{B[\ushort{X},\Delta]}(\nabla f) \norm{Z-\ushort{X}}\\
&\le \lip_{B[\ushort{X},\Delta]}(\nabla f) 2\varepsilon(\ushort{X})\\
&\le \frac{\norm{\nabla f(\ushort{X})}}{2},
\end{align*}
it holds that
\begin{equation}
\label{eq:RFDRPolakInequalityNormGradient}
\frac{1}{2} \norm{\nabla f(\ushort{X})} \le \norm{\nabla f(Z)} \le \frac{3}{2} \norm{\nabla f(\ushort{X})}.
\end{equation}
Let $X \in B[\ushort{X},\varepsilon(\ushort{X})] \cap \R_{\le r}^{m \times n}$.
Let us first consider the case where $\rank X = r$. Then,
\begin{equation*}
0 < \sigma_r(X) = \sigma_r(X) - \sigma_r(\ushort{X}) \le \norm{X-\ushort{X}} \le \varepsilon(\ushort{X}) \le \Delta.
\end{equation*}
Thus, given $X$ as input, the $\rfdr$ map considers $\hat{X} \in \proj{\R_{r-1}^{m \times n}}{X} \subset B[\ushort{X},2\varepsilon(\ushort{X})] \cap \R_{r-1}^{m \times n}$ and $\tilde{X}^\mathrm{R} \in \hyperref[algo:RFDmap]{\rfd}(\hat{X}; f, r, \ushort{\alpha}, \bar{\alpha}, \beta, c)$, where the inclusion holds because
\begin{equation*}
\norm{\hat{X}-\ushort{X}}
\le \norm{\hat{X}-X} + \norm{X-\ushort{X}}
\le \sigma_r(X) + \varepsilon(\ushort{X})
\le 2 \varepsilon(\ushort{X}).
\end{equation*}
As $X, \hat{X} \in B[\ushort{X},\rho_f(\ushort{X})]$, we have
\begin{equation}
\label{eq:RFDRPolakCostContinuous}
f(\hat{X}) \le f(X) + 2\delta(\ushort{X}).
\end{equation}
Since $B[\hat{X}, \bar{\alpha} \s(\hat{X}, f, \R_{\le r}^{m \times n})] \subseteq B[\ushort{X}, \bar{\rho}(\ushort{X})]$,  Corollary~\ref{coro:RFDmapCostDecrease} applies to $\tilde{X}^\mathrm{R}$ with the ball $B[\ushort{X},\bar{\rho}(\ushort{X})]$. The inclusion holds because, for all $Z \in B[\hat{X}, \bar{\alpha} \s(\hat{X}, f,\R_{\le r}^{m \times n})]$,
\begin{equation*}
\norm{Z-\ushort{X}}
\le \norm{Z-\hat{X}} + \norm{\hat{X}-\ushort{X}}
\le \bar{\alpha} \s(\hat{X}, f, \R_{\le r}^{m \times n}) + 2 \varepsilon(\ushort{X})
\le \bar{\alpha} \norm{\nabla f(\hat{X})} + \Delta
\le \bar{\rho}(\ushort{X}),
\end{equation*}
the last inequality following from \eqref{eq:RFDRPolakInequalityNormGradient}. Thus, for all $Y \in \hyperref[algo:RFDRmap]{\rfdr}(X; f, r, \ushort{\alpha}, \bar{\alpha}, \beta, c, \Delta)$,
\begin{align*}
f(Y)
&\le f(\tilde{X}^\mathrm{R})\\
&\le f(\hat{X}) - \frac{1}{2} c \s(\hat{X}; f, \R_{\le r}^{m \times n})^2 \min\left\{\ushort{\alpha}, 2\beta\frac{1-c}{\lip_{B[\ushort{X},\bar{\rho}(\ushort{X})]}(\nabla f)}\right\}\\
&\le f(\hat{X}) - \frac{1}{2} c \frac{\norm{\nabla f(\hat{X})}^2}{\min\{m,n\}-r+1} \min\left\{\ushort{\alpha}, 2\beta\frac{1-c}{\lip_{B[\ushort{X},\bar{\rho}(\ushort{X})]}(\nabla f)}\right\}\\
&\le f(\hat{X}) - 3 \delta(\ushort{X})\\
&\le f(X) - \delta(\ushort{X}),
\end{align*}
where the second inequality follows from \eqref{eq:RFDmapCostDecrease}, the third from the second inequality of \eqref{eq:NormProjectionTangentConeDeterminantalVariety}, the fourth from \eqref{eq:RFDRPolakInequalityNormGradient} and \eqref{eq:RFDRPolakDelta}, and the fifth from \eqref{eq:RFDRPolakCostContinuous}.
Let us now consider the case where $\rank X < r$. Let $\tilde{X} \in \hyperref[algo:RFDmap]{\rfd}(X; f, r, \ushort{\alpha}, \bar{\alpha}, \beta, c)$. Since $B[X, \bar{\alpha} \s(X, f, \R_{\le r}^{m \times n})] \subseteq B[\ushort{X}, \bar{\rho}(\ushort{X})]$, Corollary~\ref{coro:RFDmapCostDecrease} applies to $\tilde{X}$ with $B[\ushort{X},\bar{\rho}(\ushort{X})]$. The inclusion holds because, for all $Z \in B[X, \bar{\alpha} \s(X, f,\R_{\le r}^{m \times n})]$,
\begin{equation*}
\norm{Z-\ushort{X}}
\le \norm{Z-X} + \norm{X-\ushort{X}}
\le \bar{\alpha} \s(X, f, \R_{\le r}^{m \times n}) + \varepsilon(\ushort{X})
< \bar{\alpha} \norm{\nabla f(X)} + \Delta
\le \bar{\rho}(\ushort{X}),
\end{equation*}
the last inequality following from \eqref{eq:RFDRPolakInequalityNormGradient}. Thus, for all $Y \in \hyperref[algo:RFDRmap]{\rfdr}(X; f, r, \ushort{\alpha}, \bar{\alpha}, \beta, c, \Delta)$,
\begin{align*}
f(Y)
&\le f(\tilde{X})\\
&\le f(X) - \frac{1}{2} c \s(X; f, \R_{\le r}^{m \times n})^2 \min\left\{\ushort{\alpha}, 2\beta\frac{1-c}{\lip_{B[\ushort{X},\bar{\rho}(\ushort{X})]}(\nabla f)}\right\}\\
&\le f(X) - \frac{1}{2} c \frac{\norm{\nabla f(X)}^2}{\min\{m,n\}-r+1} \min\left\{\ushort{\alpha}, 2\beta\frac{1-c}{\lip_{B[\ushort{X},\bar{\rho}(\ushort{X})]}(\nabla f)}\right\}\\
&\le f(X) - 3 \delta(\ushort{X})\\
&\le f(X) - \delta(\ushort{X}),
\end{align*}
where the second inequality follows from \eqref{eq:RFDmapCostDecrease}, the third from the second inequality of \eqref{eq:NormProjectionTangentConeDeterminantalVariety}, and the fourth from \eqref{eq:RFDRPolakInequalityNormGradient} and \eqref{eq:RFDRPolakDelta}.
\end{proof}

We have thus proven the following.

\begin{theorem}
\label{thm:RFDRPolakConvergence}
Consider the sequence constructed by $\rfdr$ (Algorithm~\ref{algo:RFDR}). If this sequence is finite, then its last element is stationary, i.e., is a zero of the stationarity measure $\s(\cdot; f, \R_{\le r}^{m \times n})$ defined in \eqref{eq:StationarityMeasure}. If it is infinite, then each of its accumulation points is stationary.
\end{theorem}

\begin{corollary}
\label{coro:RFDRPolakConvergence}
Let $(X_i)_{i \in \N}$ be a sequence produced by $\rfdr$ (Algorithm~\ref{algo:RFDR}).
The sequence has at least one accumulation point if and only if $\liminf_{i \to \infty} \norm{X_i} < \infty$. For every convergent subsequence $(X_{i_k})_{k \in \N}$, $\lim_{k \to \infty} \s(X_{i_k}; f, \R_{\le r}^{m \times n}) = 0$.
If $(X_i)_{i \in \N}$ is bounded, which is the case notably if the sublevel set $\{X \in \R_{\le r}^{m \times n} \mid f(X) \le f(X_0)\}$ is bounded, then $\lim_{i \to \infty} \s(X_i; f, \R_{\le r}^{m \times n}) = 0$, and all accumulation points have the same image by $f$.
\end{corollary}

\begin{proof}
The ``if and only if'' statement is a classical result. The two limits follow from \cite[Proposition~2.12]{OlikierGallivanAbsil2022}. The final claim follows from the argument given in the proof of \cite[Theorem~65]{Polak1971}.
\end{proof}

\section{Practical implementation and computational cost}
\label{sec:PracticalImplementationComputationalCost}
In this section, we compare the computational cost of the $\ppgdr$ map \cite[Algorithm~3]{OlikierGallivanAbsil2022} with the one of the $\rfdr$ map (Algorithm~\ref{algo:RFDRmap}). The comparison is conducted based on detailed implementations of these algorithms involving only evaluations of $f$ and $\nabla f$ and some operations from linear algebra:
\begin{enumerate}
\item matrix multiplication;
\item thin QR factorization with column pivoting (see, e.g., \cite[Algorithm~5.4.1]{GolubVanLoan});
\item \emph{small scale} (truncated) SVD, i.e., the smallest dimension of the matrix to decompose is at most $2r$;
\item \emph{large scale} truncated SVD, i.e., truncated SVD that is not small scale.
\end{enumerate}
Let us recall that, in this list, only the (truncated) SVD cannot be executed within a finite number of arithmetic operations.

Detailed implementations of the $\ppgdr$ and $\rfdr$ maps are provided as Algorithms~\ref{algo:DetailedP2GDRmap} and \ref{algo:DetailedRFDRmap}, respectively. They both use Algorithm~\ref{algo:DetailedMapsZeroInput}, which is a detailed implementation of the $\ppgd$, $\ppgdr$, $\rfd$, and $\rfdr$ maps in the case where their input is $0_{m \times n}$, as a subroutine. They also rely on Algorithms~\ref{algo:DetailedP2GDmap} and \ref{algo:DetailedRFDmapSVDs}, which are detailed implementations of the $\ppgd$ and $\rfd$ maps, respectively. All those algorithms work with low-rank representations of the involved matrices as much as possible. In the rest of the section, we analyze the computational cost of the five algorithms mentioned in this paragraph.

Algorithm~\ref{algo:DetailedMapsZeroInput} involves one evaluation of $\nabla f$ (in line~\ref{algo:DetailedMapsZeroInput:TruncatedSVD}), at most
\begin{equation}
\label{eq:MaxNumIterationsDetailedMapsZeroInput}
1+\max\left\{0, \left\lfloor\ln\bigg(\frac{2\beta(1-c)}{\alpha\lip_{B[0_{m \times n},\alpha\s(0_{m \times n}; f, \R_{\le r}^{m \times n})]}(\nabla f)}\bigg)/\ln \beta\right\rfloor\right\}
\end{equation}
evaluations of $f$ (in the while loop), and one large scale truncated SVD (in line~\ref{algo:DetailedMapsZeroInput:TruncatedSVD}). The upper bound \eqref{eq:MaxNumIterationsDetailedMapsZeroInput} follows from \eqref{eq:RFDmapMaxNumIterationsWhile}.

\begin{algorithm}[H]
\caption{Detailed $\ppgd$, $\ppgdr$, $\rfd$, and $\rfdr$ maps given zero as input}
\label{algo:DetailedMapsZeroInput}
\begin{algorithmic}[1]
\Require
$(f, r, \alpha, \beta, c)$ where $f : \R^{m \times n} \to \R$ is differentiable with $\nabla f$ locally Lipschitz continuous, $r < \min\{m,n\}$ is a positive integer, $\alpha \in (0,\infty)$, $\beta, c \in (0,1)$, and $\s(0_{m \times n}; f, \R_{\le r}^{m \times n}) > 0$.
\Output
$(\tilde{U}, \tilde{\Sigma}, \tilde{V})$ where $\tilde{U}\tilde{\Sigma}\tilde{V}^\tp \in \R_{\tilde{r}}^{m \times n}$ is an SVD, $\tilde{r} \in \{1, \dots, r\}$, and $\tilde{U}\tilde{\Sigma}\tilde{V}^\tp \in \ppgd(0_{m \times n}; f, r, \alpha, \alpha, \beta, c) = \hyperref[algo:RFDmap]{\rfd}(0_{m \times n}; f, r, \alpha, \alpha, \beta, c) = \ppgdr(0_{m \times n}; f, r, \alpha, \alpha, \beta, c, \Delta) = \hyperref[algo:RFDRmap]{\rfdr}(0_{m \times n}; f, r, \alpha, \alpha, \beta, c, \Delta)$.

\State
Compute a truncated SVD $\tilde{U} \Sigma \tilde{V}^\tp \in \R_{\le r}^{m \times n}$ of $-\nabla f(0_{m \times n})$;
\label{algo:DetailedMapsZeroInput:TruncatedSVD}
\State
$s \gets \norm{\Sigma}^2$;
\While
{$f(\alpha \tilde{U} \Sigma \tilde{V}^\tp) > f(0_{m \times n}) - c \, \alpha \, s$}
\State
$\alpha \gets \alpha \beta$;
\EndWhile
\State
$\tilde{\Sigma} \gets \alpha \Sigma$;
\end{algorithmic}
\end{algorithm}

\begin{algorithm}[H]
\caption{Detailed $\ppgd$ map (based on \cite[Algorithms~2 and 3]{SchneiderUschmajew2015})}
\label{algo:DetailedP2GDmap}
\begin{algorithmic}[1]
\Require
$(f, r, \alpha, \beta, c)$ where $f : \R^{m \times n} \to \R$ is differentiable with $\nabla f$ locally Lipschitz continuous, $r < \min\{m,n\}$ is a positive integer, $\alpha \in (0,\infty)$, and $\beta, c \in (0,1)$.
\Input
$(U, \Sigma, V)$ where $U \Sigma V^\tp \in \R_{\ushort{r}}^{m \times n}$ is an SVD, $\ushort{r} \in \{1, \dots, r\}$, and $\s(U \Sigma V^\tp; f, \R_{\le r}^{m \times n}) > 0$.
\Output
$(\tilde{U}, \tilde{\Sigma}, \tilde{V})$ where $\tilde{U}\tilde{\Sigma}\tilde{V}^\tp \in \R_{\tilde{r}}^{m \times n}$ is an SVD, $\tilde{r} \in \{1, \dots, r\}$, and $\tilde{U}\tilde{\Sigma}\tilde{V}^\tp \in \ppgd(U \Sigma V^\tp; f, r, \alpha, \alpha, \beta, c)$.

\State
$\bar{G} \gets -\nabla f(U \Sigma V^\tp)$;
$\hat{G}_1 \gets U^\tp \bar{G}$;
$\hat{G}_2 \gets \bar{G} V$;
$\hat{G} \gets \hat{G}_1 V$;
$\hat{\hat{G}}_1 \gets U \hat{G}$;
$\hat{\hat{G}}_2 \gets \hat{G} V^\tp$;
\label{algo:DetailedP2GDmap:GradientEvaluation}
\State
Compute QR factorizations with column pivoting $\hat{G}_2-\hat{\hat{G}}_1 = U_\perp R_1$ and $\hat{G}_1^\tp-\hat{\hat{G}}_2^\tp = V_\perp R_2$ where $U_\perp \in \st(r_1,m)$, $V_\perp \in \st(r_2,n)$, $R_1 \in \R_*^{r_1 \times \ushort{r}}$, and $R_2 \in \R_*^{r_2 \times \ushort{r}}$;
\label{algo:DetailedP2GDmap:QRColumnPivoting1}
\State
$s \gets \norm{\hat{G}}^2 + \norm{R_1}^2 + \norm{R_2}^2$;
\If
{$\ushort{r} = r$}
\State
Compute a truncated SVD $\hat{U} \hat{\Sigma} \hat{V}^\tp \in \R_{\le r}^{\ushort{r}+r_1 \times \ushort{r}+r_2}$ of $\begin{bmatrix} \Sigma+\alpha\hat{G} & \alpha R_2^\tp \\ \alpha R_1 & 0_{r_1 \times r_2} \end{bmatrix}$;
\label{algo:DetailedP2GDmap:SmallTruncatedSVD1}
\While
{$f([U \; U_\perp] \hat{U} \hat{\Sigma} \hat{V}^\tp [V \; V_\perp]^\tp) > f(U \Sigma V^\tp) - c \, \alpha \, s$}
\State
$\alpha \gets \alpha \beta$;
\State
Compute a truncated SVD $\hat{U} \hat{\Sigma} \hat{V}^\tp \in \R_{\le r}^{\ushort{r}+r_1 \times \ushort{r}+r_2}$ of $\begin{bmatrix} \Sigma+\alpha\hat{G} & \alpha R_2^\tp \\ \alpha R_1 & 0_{r_1 \times r_2} \end{bmatrix}$;
\label{algo:DetailedP2GDmap:SmallTruncatedSVD2}
\EndWhile
\State
$\tilde{\Sigma} \gets \hat{\Sigma}$; $\tilde{U} \gets [U \; U_\perp] \hat{U}$; $\tilde{V} \gets [V \; V_\perp] \hat{V}$;
\label{algo:DetailedP2GDmap:rend}
\Else
\State
$\tilde{G} \gets \bar{G}-U\hat{G}_1+(\hat{\hat{G}}_1-\hat{G}_2)V^\tp$;
\If
{$\tilde{G} = 0_{m \times n}$}
\State
Repeat lines \ref{algo:DetailedP2GDmap:SmallTruncatedSVD1} to \ref{algo:DetailedP2GDmap:rend};
\Else
\State
Compute a truncated SVD $\ushort{U}_\perp \ushort{\Sigma} \ushort{V}_\perp^\tp \in \R_{\le r-\ushort{r}}^{m \times n}$ of $\tilde{G} \in \R_{\le \min\{m,n\}-\ushort{r}}^{m \times n}$;
\label{algo:DetailedP2GDmap:TruncatedSVD}
\State
$s \gets s+\norm{\ushort{\Sigma}}^2$; $r_0 \gets \rank \ushort{\Sigma}$;
\State
Compute QR factorizations with column pivoting $[U_\perp \; \ushort{U}_\perp] = [U_\perp \; \bar{U}_\perp] \begin{bmatrix} I_{r_1} & R_{1,1} \\ 0_{r_3 \times r_1} & R_{1,2} \end{bmatrix}$ and $[V_\perp \; \ushort{V}_\perp] = [V_\perp \; \bar{V}_\perp] \begin{bmatrix} I_{r_2} & R_{2,1} \\ 0_{r_4 \times r_2} & R_{2,2} \end{bmatrix}$ where $\bar{U}_\perp \in \st(r_3,m)$, $\bar{V}_\perp \in \st(r_4,n)$, $R_{1,1} \in \R^{r_1 \times r_0}$, $R_{1,2} \in \R_*^{r_3 \times r_0}$, $R_{2,1} \in \R^{r_2 \times r_0}$, and $R_{2,2} \in \R_*^{r_4 \times r_0}$;
\label{algo:DetailedP2GDmap:QRColumnPivoting2}
\State
Compute a truncated SVD $\hat{U} \hat{\Sigma} \hat{V}^\tp \in \R_{\le r}^{\ushort{r}+r_1+r_3 \times \ushort{r}+r_2+r_4}$ of $\begin{bmatrix} \Sigma+\alpha\hat{G} & \alpha R_2^\tp & 0_{\ushort{r} \times r_4} \\ \alpha R_1 & \alpha R_{1,1} \ushort{\Sigma} R_{2,1}^\tp & \alpha R_{1,1} \ushort{\Sigma} R_{2,2}^\tp \\ 0_{r_3 \times \ushort{r}} & \alpha R_{1,2} \ushort{\Sigma} R_{2,1}^\tp & \alpha R_{1,2} \ushort{\Sigma} R_{2,2}^\tp \end{bmatrix}$;
\label{algo:DetailedP2GDmap:SmallTruncatedSVD3}
\While
{$f([U \; U_\perp \; \bar{U}_\perp] \hat{U} \hat{\Sigma} \hat{V}^\tp [V \; V_\perp \; \bar{V}_\perp]^\tp) > f(U \Sigma V^\tp) - c \, \alpha \, s$}
\State
$\alpha \gets \alpha \beta$;
\State
Compute a truncated SVD $\hat{U} \hat{\Sigma} \hat{V}^\tp \in \R_{\le r}^{\ushort{r}+r_1+r_3 \times \ushort{r}+r_2+r_4}$ of $\begin{bmatrix} \Sigma+\alpha\hat{G} & \alpha R_2^\tp & 0_{\ushort{r} \times r_4} \\ \alpha R_1 & \alpha R_{1,1} \ushort{\Sigma} R_{2,1}^\tp & \alpha R_{1,1} \ushort{\Sigma} R_{2,2}^\tp \\ 0_{r_3 \times \ushort{r}} & \alpha R_{1,2} \ushort{\Sigma} R_{2,1}^\tp & \alpha R_{1,2} \ushort{\Sigma} R_{2,2}^\tp \end{bmatrix}$;
\label{algo:DetailedP2GDmap:SmallTruncatedSVD4}
\EndWhile
\State
$\tilde{\Sigma} \gets \hat{\Sigma}$; $\tilde{U} \gets [U \; U_\perp \; \bar{U}_\perp] \hat{U}$; $\tilde{V} \gets [V \; V_\perp \; \bar{V}_\perp] \hat{V}$;
\EndIf
\EndIf
\end{algorithmic}
\end{algorithm}

Algorithm~\ref{algo:DetailedP2GDmap} involves one evaluation of $\nabla f$ (in line~\ref{algo:DetailedP2GDmap:GradientEvaluation}), at most
\begin{equation}
\label{eq:MaxNumIterationsDetailedP2GDmap}
1+\max\left\{0, \left\lfloor\ln\bigg(\frac{\beta(1-c)}{\alpha\kappa_{B[X,\alpha\s(X; f, \R_{\le r}^{m \times n})]}(X; f, \alpha)}\bigg)/\ln \beta\right\rfloor\right\}
\end{equation}
evaluations of $f$ and small scale truncated SVDs (in the while loop that is executed), several matrix multiplications, at most four QR factorizations with column pivoting (in lines~\ref{algo:DetailedP2GDmap:QRColumnPivoting1} and \ref{algo:DetailedP2GDmap:QRColumnPivoting2}), and at most one large scale truncated SVD (in line~\ref{algo:DetailedP2GDmap:TruncatedSVD}). The upper bound \eqref{eq:MaxNumIterationsDetailedP2GDmap} follows from \cite[Corollary~3.2]{OlikierGallivanAbsil2022}.
Let us point out that, if $\hat{G}_1 = \hat{\hat{G}}_2$ or $\hat{G}_2 = \hat{\hat{G}}_1$ in line~\ref{algo:DetailedP2GDmap:QRColumnPivoting1}, then $r_1$ or $r_2$ is zero and empty matrices appear in the algorithm.

Based on the analysis of Algorithms~\ref{algo:DetailedMapsZeroInput} and \ref{algo:DetailedP2GDmap}, we analyze Algorithm~\ref{algo:DetailedP2GDRmap}. As mentioned in Section~\ref{sec:Introduction}, in the worst case, i.e., if the input has rank $r$ and its largest singular value is smaller than or equal to $\Delta$, then Algorithm~\ref{algo:DetailedP2GDRmap} calls Algorithm~\ref{algo:DetailedP2GDmap} $r$ times and Algorithm~\ref{algo:DetailedMapsZeroInput} once. Besides matrix multiplications, this requires $r+1$ evaluations $\nabla f$, at most \eqref{eq:MaxNumIterationsDetailedMapsZeroInput} plus $r$ times \eqref{eq:MaxNumIterationsDetailedP2GDmap} evaluations of $f$, at most $4r-2$ QR factorizations with column pivoting, at most $r$ times \eqref{eq:MaxNumIterationsDetailedP2GDmap} small scale truncated SVDs, and $r$ large scale truncated SVDs.

\begin{algorithm}[H]
\caption{Detailed $\ppgdr$ map}
\label{algo:DetailedP2GDRmap}
\begin{algorithmic}[1]
\Require
$(f, r, \alpha, \beta, c, \Delta)$ where $f : \R^{m \times n} \to \R$ is differentiable with $\nabla f$ locally Lipschitz continuous, $r < \min\{m,n\}$ is a positive integer, $\alpha \in (0,\infty)$, $\beta, c \in (0,1)$, and $\Delta \in (0,\infty)$.
\Input
$(U, \Sigma, V)$ where $U \Sigma V^\tp \in \R_{\ushort{r}}^{m \times n}$ is an SVD, $\ushort{r} \in \{1, \dots, r\}$, and $\s(U \Sigma V^\tp; f, \R_{\le r}^{m \times n}) > 0$.
\Output
$(\tilde{U}, \tilde{\Sigma}, \tilde{V})$ where $\tilde{U}\tilde{\Sigma}\tilde{V}^\tp \in \R_{\tilde{r}}^{m \times n}$ is an SVD, $\tilde{r} \in \{1, \dots, r\}$, and $\tilde{U}\tilde{\Sigma}\tilde{V}^\tp \in \ppgdr(U \Sigma V^\tp; f, r, \alpha, \alpha, \beta, c, \Delta)$.

\State
$(\tilde{U}, \tilde{\Sigma}, \tilde{V}) \gets \text{Algorithm~\ref{algo:DetailedP2GDmap}}(\Sigma, U, V; f, r, \alpha, \beta, c)$;
\State
$r_\Delta \gets |\{j \in \{1, \dots, \ushort{r}\} \mid \Sigma(j,j) > \Delta\}|$;
\For
{$j \in \{1, \dots, \ushort{r}-r_\Delta\}$}
\If
{$j < \ushort{r}$}
\State
$(\tilde{U}^\mathrm{R}, \tilde{\Sigma}^\mathrm{R}, \tilde{V}^\mathrm{R}) \gets \text{Algorithm~\ref{algo:DetailedP2GDmap}}(\Sigma(1:\ushort{r}-j, 1:\ushort{r}-j), U(:, 1:\ushort{r}-j), V(:, 1:\ushort{r}-j); f, r, \alpha, \beta, c)$;
\Else
\State
$(\tilde{U}^\mathrm{R}, \tilde{\Sigma}^\mathrm{R}, \tilde{V}^\mathrm{R}) \gets \text{Algorithm~\ref{algo:DetailedMapsZeroInput}}(f, r, \alpha, \beta, c)$;
\EndIf
\If
{$f(\tilde{U}^\mathrm{R} \tilde{\Sigma}^\mathrm{R} (\tilde{V}^\mathrm{R})^\tp) < f(\tilde{U} \tilde{\Sigma} \tilde{V}^\tp)$}
\State
$(\tilde{U}, \tilde{\Sigma}, \tilde{V}) \gets (\tilde{U}^\mathrm{R}, \tilde{\Sigma}^\mathrm{R}, \tilde{V}^\mathrm{R})$;
\EndIf
\EndFor
\end{algorithmic}
\end{algorithm}

Algorithm~\ref{algo:DetailedRFDmapSVDs} involves one evaluation of $\nabla f$ (in line~\ref{algo:DetailedRFDmapSVDs:GradientEvaluation}), at most
\begin{equation}
\label{eq:MaxNumIterationsDetailedRFDmapSVDs}
1+\max\left\{0, \left\lfloor\ln\bigg(\frac{2\beta(1-c)}{\alpha\lip_{B[X,\alpha\s(X; f, \R_{\le r}^{m \times n})]}(\nabla f)}\bigg)/\ln \beta\right\rfloor\right\}
\end{equation}
evaluations of $f$ (in the while loop that is executed), several matrix multiplications, one small scale SVD (in line~\ref{algo:DetailedRFDmapSVDs:SmallSVD1} or \ref{algo:DetailedRFDmapSVDs:SmallSVD2}), and at most one large scale truncated SVD (in line~\ref{algo:DetailedRFDmapSVDs:TruncatedSVD}). The upper bound \eqref{eq:MaxNumIterationsDetailedRFDmapSVDs} follows from \eqref{eq:RFDmapMaxNumIterationsWhile}.

Let us mention that, if one is only interested in the $\rfd$ map, then the SVD computed in line~\ref{algo:DetailedRFDmapSVDs:SmallSVD1} or \ref{algo:DetailedRFDmapSVDs:SmallSVD2} is not necessary; it can be replaced, e.g., by a QR factorization with column pivoting. However, Algorithm~\ref{algo:DetailedRFDmapSVDs} uses an SVD because it is a subroutine of Algorithm~\ref{algo:DetailedRFDRmap} which requires an SVD as input.

\begin{algorithm}[H]
\caption{Detailed $\rfd$ map via SVDs (based on \cite[Algorithms~2 and 4]{SchneiderUschmajew2015})}
\label{algo:DetailedRFDmapSVDs}
\begin{algorithmic}[1]
\Require
$(f, r, \alpha, \beta, c)$ where $f : \R^{m \times n} \to \R$ is differentiable with $\nabla f$ locally Lipschitz continuous, $r < \min\{m,n\}$ is a positive integer, $\alpha \in (0,\infty)$, and $\beta, c \in (0,1)$.
\Input
$(U, \Sigma, V)$ where $U \Sigma V^\tp \in \R_{\ushort{r}}^{m \times n}$ is an SVD, $\ushort{r} \in \{1, \dots, r\}$, and $\s(U \Sigma V^\tp; f, \R_{\le r}^{m \times n}) > 0$.
\Output
$(\tilde{U}, \tilde{\Sigma}, \tilde{V})$ where $\tilde{U}\tilde{\Sigma}\tilde{V}^\tp \in \R_{\tilde{r}}^{m \times n}$ is an SVD, $\tilde{r} \in \{1, \dots, r\}$, and $\tilde{U}\tilde{\Sigma}\tilde{V}^\tp \in \hyperref[algo:RFDmap]{\rfd}(U \Sigma V^\tp; f, r, \alpha, \alpha, \beta, c)$.

\State
$\bar{G} \gets -\nabla f(U \Sigma V^\tp)$;
$\hat{G}_1 \gets U^\tp \bar{G}$;
$\hat{G}_2 \gets \bar{G} V$;
$s_1 \gets \norm{\hat{G}_1}^2$;
$s_2 \gets \norm{\hat{G}_2}^2$;
\label{algo:DetailedRFDmapSVDs:GradientEvaluation}
\If
{$s_1 \ge s_2$}
\State
$\hat{X}_2 \gets \Sigma V^\tp$;
\label{algo:DetailedRFDmapSVDs:G1}
\If
{$\ushort{r} = r$}
\While
{$f(U(\hat{X}_2+\alpha \hat{G}_1)) > f(U \Sigma V^\tp) - c \, \alpha \, s_1$}
\label{algo:DetailedRFDmapSVDs:r}
\State
$\alpha \gets \alpha \beta$;
\EndWhile
\State
Compute an SVD $\hat{U} \tilde{\Sigma} \tilde{V}^\tp \in \R_{\tilde{r}}^{\ushort{r} \times n}$ of $\hat{X}_2+\alpha \hat{G}_1$;
\label{algo:DetailedRFDmapSVDs:SmallSVD1}
\State
$\tilde{U} \gets U \hat{U}$;
\label{algo:DetailedRFDmapSVDs:rend}
\Else
\State
$\tilde{G} \gets \bar{G}-U\hat{G}_1+(U(\hat{G}_1V)-\hat{G}_2)V^\tp$;
\If
{$\tilde{G} = 0_{m \times n}$}
\State
Repeat lines \ref{algo:DetailedRFDmapSVDs:r} to \ref{algo:DetailedRFDmapSVDs:rend};
\Else
\State
Compute a truncated SVD $\ushort{U}_\perp \ushort{\Sigma} \ushort{V}_\perp^\tp \in \R_{\le r-\ushort{r}}^{m \times n}$ of $\tilde{G} \in \R_{\le \min\{m,n\}-\ushort{r}}^{m \times n}$;
\label{algo:DetailedRFDmapSVDs:TruncatedSVD}
\State
$s_1 \gets s_1+\norm{\ushort{\Sigma}}^2$; $r_0 \gets \rank \ushort{\Sigma}$;
\While
{$f([U \; \ushort{U}_\perp] \begin{bmatrix} \hat{X}_2 + \alpha \hat{G}_1 \\ \alpha \ushort{\Sigma} \ushort{V}_\perp^\tp \end{bmatrix}) > f(USV^\tp) - c \, \alpha \, s_1$}
\State
$\alpha \gets \alpha \beta$;
\EndWhile
\State
Compute an SVD $\hat{U} \tilde{\Sigma} \tilde{V}^\tp \in \R_{\tilde{r}}^{\ushort{r}+r_0 \times n}$ of $\begin{bmatrix} \hat{X}_2 + \alpha \hat{G}_1 \\ \alpha \ushort{\Sigma} \ushort{V}_\perp^\tp \end{bmatrix}$;
\label{algo:DetailedRFDmapSVDs:SmallSVD2}
\State
$\tilde{U} \gets [U \; \ushort{U}_\perp] \hat{U}$;
\EndIf
\EndIf
\label{algo:DetailedRFDmapSVDs:G1end}
\Else
\State
Repeat mutatis mutandis lines~\ref{algo:DetailedRFDmapSVDs:G1} to \ref{algo:DetailedRFDmapSVDs:G1end} with $s_1$ replaced by $s_2$;
\EndIf
\end{algorithmic}
\end{algorithm}

In the worst case, i.e., if the input has rank $r$ and its smallest singular value is smaller than or equal to $\Delta$, then, assuming that $r > 1$, Algorithm~\ref{algo:DetailedRFDRmap} calls Algorithm~\ref{algo:DetailedRFDmapSVDs} twice. Besides matrix multiplications, this requires two evaluations of $\nabla f$, at most two times \eqref{eq:MaxNumIterationsDetailedRFDmapSVDs} evaluations of $f$, two small scale SVDs, and one large scale truncated SVD.

\begin{algorithm}[H]
\caption{Detailed $\rfdr$ map}
\label{algo:DetailedRFDRmap}
\begin{algorithmic}[1]
\Require
$(f, r, \alpha, \beta, c, \Delta)$ where $f : \R^{m \times n} \to \R$ is differentiable with $\nabla f$ locally Lipschitz continuous, $r < \min\{m,n\}$ is a positive integer, $\alpha \in (0,\infty)$, $\beta, c \in (0,1)$, and $\Delta \in (0,\infty)$.
\Input
$(U, \Sigma, V)$ where $U \Sigma V^\tp \in \R_{\ushort{r}}^{m \times n}$ is an SVD, $\ushort{r} \in \{1, \dots, r\}$, and $\s(U \Sigma V^\tp; f, \R_{\le r}^{m \times n}) > 0$.
\Output
$(\tilde{U}, \tilde{\Sigma}, \tilde{V})$ where $\tilde{U}\tilde{\Sigma}\tilde{V}^\tp \in \R_{\tilde{r}}^{m \times n}$ is an SVD, $\tilde{r} \in \{1, \dots, r\}$, and $\tilde{U}\tilde{\Sigma}\tilde{V}^\tp \in \hyperref[algo:RFDRmap]{\rfdr}(U \Sigma V^\tp; f, r, \alpha, \alpha, \beta, c, \Delta)$.

\State
$(\tilde{U}, \tilde{\Sigma}, \tilde{V}) \gets \text{Algorithm~\ref{algo:DetailedRFDmapSVDs}}(U, \Sigma, V; f, r, \alpha, \beta, c)$;
\If
{$\ushort{r} = r$ and $\Sigma(r,r) \le \Delta$}
\If
{$r > 1$}
\State
$(\tilde{U}^\mathrm{R}, \tilde{\Sigma}^\mathrm{R}, \tilde{V}^\mathrm{R}) \gets \text{Algorithm~\ref{algo:DetailedRFDmapSVDs}}(U(:, 1:r-1), \Sigma(1:r-1, 1:r-1), V(:, 1:r-1); f, r, \alpha, \beta, c)$;
\Else
\State
$(\tilde{U}^\mathrm{R}, \tilde{\Sigma}^\mathrm{R}, \tilde{V}^\mathrm{R}) \gets \text{Algorithm~\ref{algo:DetailedMapsZeroInput}}(f, r, \alpha, \beta, c)$;
\EndIf
\If
{$f(\tilde{U}^\mathrm{R} \tilde{\Sigma}^\mathrm{R} (\tilde{V}^\mathrm{R})^\tp) < f(\tilde{U} \tilde{\Sigma} \tilde{V}^\tp)$}
\State
$(\tilde{U}, \tilde{\Sigma}, \tilde{V}) \gets (\tilde{U}^\mathrm{R}, \tilde{\Sigma}^\mathrm{R}, \tilde{V}^\mathrm{R})$;
\EndIf
\EndIf
\end{algorithmic}
\end{algorithm}

Table~\ref{tab:CostP2GDRvsRFDR} summarizes the operations required by Algorithms~\ref{algo:DetailedP2GDRmap} and \ref{algo:DetailedRFDRmap} in the case where $r > 1$, the input has rank $r$, and its largest singular value is smaller than or equal to $\Delta$. We observe that Algorithm~\ref{algo:DetailedP2GDRmap} requires more of each operation, except perhaps the evaluation of $f$, than Algorithm~\ref{algo:DetailedRFDRmap}.

\begin{table}[h]
\begin{center}
\begin{tabular}{*{3}{l}}
\hline
Operation & Algorithm~\ref{algo:DetailedP2GDRmap} ($\ppgdr$) & Algorithm~\ref{algo:DetailedRFDRmap} ($\rfdr$)\\
\hline\hline
evaluation of $f$ & up to $\eqref{eq:MaxNumIterationsDetailedMapsZeroInput} + r\cdot \eqref{eq:MaxNumIterationsDetailedP2GDmap}$ & up to $2 \cdot \eqref{eq:MaxNumIterationsDetailedRFDmapSVDs}$ \\
\hline
evaluation of $\nabla f$ & $r+1$ & $2$ \\
\hline
QR factorization with c. p. & up to $4r-2$ & $0$ \\
\hline
small scale (truncated) SVD & $r$ times \eqref{eq:MaxNumIterationsDetailedP2GDmap} & $2$ \\
\hline
large scale truncated SVD & $r$ & $1$ \\
\hline
\end{tabular}
\end{center}
\caption{Operations required by Algorithms~\ref{algo:DetailedP2GDRmap} and \ref{algo:DetailedRFDRmap} in the case where $r > 1$, the input has rank $r$, and its largest singular value is smaller than or equal to $\Delta$.}
\label{tab:CostP2GDRvsRFDR}
\end{table}

\section{Conclusion}
\label{sec:Conclusion}
We close this work with three concluding remarks.
\begin{enumerate}
\item As in \cite{LevinKileelBoumal2022} and \cite{OlikierGallivanAbsil2022}, everything said in this paper remains true if $f$ is only defined on an open subset of $\R^{m \times n}$ containing $\R_{\le r}^{m \times n}$.

\item We stated in Section~\ref{sec:Introduction} that $\rfd$ can follow apocalypses. It can indeed be verified that $\rfd$ follows the apocalypses described in \cite[\S 2.2]{LevinKileelBoumal2022} and \cite[\S 3.2]{OlikierGallivanAbsil2022Comparison}. In fact, for these two instances of \eqref{eq:LowRankOpti}, $\rfd$ and $\rfdr$ respectively produce the same sequences of iterates as $\ppgd$ and $\ppgdr$. The main reason is that, for each of the two instances, if $(X_i)_{i \in \N}$ denotes the sequence produced by $\ppgd$ or $\ppgdr$, then, for all $i \in \N$,
\begin{equation}
\label{eq:ApproximateProjectionCoincidesWithProjection}
\proj{\restancone{\R_{\le r}^{m \times n}}{X_i}}{-\nabla f(X_i)} =
\proj{\tancone{\R_{\le r}^{m \times n}}{X_i}}{-\nabla f(X_i)}.
\end{equation}

\item As $\rfd$, $\rfdr$ requires the computation of at most one large scale truncated SVD per iteration. An open question is whether it is possible to design a first-order optimization algorithm on $\R_{\le r}^{m \times n}$ offering the same convergence properties as $\rfdr$ without involving any large scale truncated SVD.
\end{enumerate}

\bibliographystyle{siamplain}
\bibliography{golikier_bib}
\end{document}